\documentclass{amsart}
\usepackage{enumerate}
\usepackage{amsthm}
\usepackage{amssymb}
\usepackage{amsmath}
\usepackage{graphicx}
\usepackage{mathrsfs}
\usepackage{mathtools}
\usepackage{subfig}
\usepackage{tabularx}
\newtheorem{theorem}{Theorem}[section]
\newtheorem{prop}{Proposition}[section]
\newtheorem{cor}{Corollary}[section]
\newtheorem*{theo*}{Theorem}
\newtheorem{lm}{Lemma}[section]

\theoremstyle{definition}
\newtheorem{defin}{Definition}[section]
\newtheorem*{rem*}{Remark}
\newtheorem{rem}{Remark}[section]

\begin{document}
	\title[Quasi-Geodesics and Backward Orbits]{Quasi-Geodesics and Backward Orbits Under Semigroups of Holomorphic Functions}
	\author{Konstantinos Zarvalis}
	
	\address{Department of Mathematics, Aristotle University of Thessaloniki, 54124, Thessaloniki, Greece}
	\email{zarkonath@math.auth.gr}
	\maketitle
	\begin{abstract}
		We explore two properties of backward orbits under semigroups of holomorphic self-maps in the unit disk. First, we prove that regular backward orbits are quasi-geodesics for the hyperbolic distance of the unit disk. Then, we show that backward orbits satisfy a useful property, this time in Euclidean terms. 
	\end{abstract}
	\section{Introduction}
	\textit{One-parameter continuous semigroups of holomorphic functions in the unit disk} $\mathbb{D}$, or, from now on, \textit{semigroups in} $\mathbb{D}$, have been studied extensively in recent years. They were first introduced by Berkson and Porta in \cite{berkson} and their theory is presented in several books, such as \cite{abate}, \cite{elin}, and more recently in the monograph \cite{book}.
	
	A semigroup in $\mathbb{D}$ is a family $(\phi_t)$, $t\ge0$, of holomorphic self-maps of the unit disk that satisfies the following three properties:
	
	\begin{enumerate}[(i)]
		\item $\phi_0$ is the identity in $\mathbb{D}$,
		\item $\phi_{t+s}(z)=\phi_t(\phi_s(z))$, for all $z\in\mathbb{D}$ and all $t,s\ge0$,
		\item $\lim\limits_{t\to s}\phi_t(z)=\phi_s(z)$, for all $z\in\mathbb{D}$ and all $s\ge0$.
	\end{enumerate}
	
	All the elements of a semigroup have the same dynamic behavior (\cite[Remark 12.1.2]{book}). In particular, if, for some $t_0>0$, the function $\phi_{t_0}$ has no fixed points inside the unit disk, then the same is true for any other element of the semigroup, other than $\phi_0$. Such semigroups, without any fixed points, are called \textit{non-elliptic}. For a non-elliptic semigroup $(\phi_t)$, there always exists a unique point $\tau\in\partial\mathbb{D}$ such that
	$$\lim\limits_{t\to+\infty}\phi_t(z)=\tau,$$
	for all $z\in\mathbb{D}$ (\cite[Theorem 8.3.1]{book}). This point $\tau$ is called the \textit{Denjoy-Wolff} point of the semigroup.
	\par A very important tool, connected with the study of a semigroup $(\phi_t)$, is the so-called \textit{Koenigs function} $h$ of the semigroup. More specifically and for non-elliptic semigroups, this function is the unique conformal mapping $h:\mathbb{D}\to\mathbb{C}$ satisfying $h(0)=0$, such that
	$$\phi_t(z)=h^{-1}(h(z)+t), \quad z\in\mathbb{D}, \quad t\ge0.$$
	The simply connected domain $\Omega:=h(\mathbb{D})$ is called the \textit{associated planar domain} of the semigroup.
	\par A direct implication of the above relation concerning the Koenigs function is the fact that $\Omega$ is \textit{convex in the positive direction} (also known as \textit{starlike at infinity}). This means that $\{w+t:t\ge0\}\subset\Omega$, for every $w\in\Omega$.
	\par A first distinction within the class of non-elliptic semigroups was given in \cite{flows}. If $\Omega$ is contained in a horizontal strip, then the semigroup is called \textit{hyperbolic}. Otherwise, it is called \textit{parabolic}.
	\par Let $z\in\mathbb{D}$. Then the curve $\gamma_z:[0,+\infty)\to\mathbb{D}$, with $\gamma_z(t)=\phi_t(z), t\ge0,$ is called the \textit{trajectory} of $z$. Later on, we might refer to $\gamma_z$ as the \textit{forward trajectory} of $z$ to avoid confusion. It is clear, by the definition of the Denjoy-Wolff point $\tau$, that every trajectory converges to $\tau$.
	\par Given $z\in\mathbb{D}$, one can also study the \textit{backward trajectory} of $z$ until it reaches a point of the unit circle. In particular, if such a trajectory can be defined for all negative ``times'', then we have a \textit{backward orbit}. More formally, if $(\phi_t)$ is a semigroup in $\mathbb{D}$, then a continuous curve $\gamma:[0,+\infty)\to\mathbb{D}$ is called a backward orbit if for every $t\in(0,+\infty)$ and for every $s\in[0,t],$
	$$\phi_s(\gamma(t))=\gamma(t-s).$$
	Backward orbits for one-parameter semigroups in $\mathbb{D}$ were initially studied in \cite{flows} and then in \cite{elin2}. All the information concerning backward orbits needed in this article will be mostly drawn from \cite{back} (see also \cite[Chapter 13]{book}).
	\par Leaving aside the semigroups, for a moment, we now need to provide some information with regard to hyperbolic geometry in order to be able to state our first theorem. To begin with, 
	we denote by $\lambda_\mathbb{D}$ the \textit{hyperbolic density} of the unit disk, which is given by the relation
	$$\lambda_\mathbb{D}(z)|dz|=\frac{|dz|}{1-|z|^2}.$$
	 Then, the \textit{hyperbolic distance} between any two points $z, w$ of the unit disk is defined as
	$$k_\mathbb{D}(z,w)=\inf\limits_\gamma\int\limits_\gamma\lambda_\mathbb{D},$$
	where the infimum is taken over all smooth curves $\gamma$ that join $z$ and $w$ inside the unit disk. Now, let $\Omega\subsetneq\mathbb{C}$ be a simply connected domain and $f:\Omega\to\mathbb{D}$ a corresponding Riemann map. We denote by $\lambda_\Omega$ the hyperbolic density of $\Omega$ given by 
	$$\lambda_\Omega(z)=\lambda_\mathbb{D}(f(z))\cdot|f'(z)|, \quad z\in\Omega,$$
	and by $k_\Omega$ the hyperbolic distance in $\Omega$ given by
	$$k_\Omega(z,w)=k_\mathbb{D}(f(z),f(w)), \quad z,w\in\Omega.$$
	It can be verified that this definition is indeed independent of the choice of the Riemann map $f$.  In addition, given a simply connected domain $\Omega$ and a smooth curve $\gamma:(a,b)\to\Omega$ (where $-\infty\le a<b\le+\infty$), we denote the \textit{hyperbolic length} of $\gamma$ between the points $\gamma(t_1)$ and $\gamma(t_2)$, where $a<t_1\le t_2<b$, by $l_\Omega(\gamma;[t_1,t_2])$. The hyperbolic length is given by the relation
	$$l_\Omega(\gamma;[t_1,t_2])=\int\limits_\gamma\lambda_\Omega(z)|dz|=\int\limits_{t_1}^{t_2}\lambda_\Omega(\gamma(t))|\gamma'(t)|dt.$$
	The definition of the hyperbolic length can be extended in a natural way in case the domain of $\gamma$ is of the form $[a,b],[a,b)$ or $(a,b]$.
	\par As usual, a \textit{geodesic} for the hyperbolic distance is a smooth curve such that the hyperbolic length among any two points of the curve coincides with the hyperbolic distance between these two points. For the purposes of this article, we need the notion of a \textit{quasi-geodesic} as well. We mostly follow \cite[Chapter 6]{book} for the necessary information concerning quasi-geodesics. For a more detailed study on quasi-geodesics, the reader may refer to \cite{quasi}.
	\begin{defin}
		Let $\gamma:[0,+\infty)\to\mathbb{D}$ be a Lipschitz curve that satisfies $\lim\limits_{t\to+\infty}k_\mathbb{D}(\gamma(0),\gamma(t))=+\infty$ (i.e. the curve tends to the unit circle). Then, $\gamma$ is called an \textit{($A,B$)-quasi-geodesic} for the hyperbolic distance $k_\mathbb{D}$ of the unit disk provided there exist absolute positive constants $A\ge1$ and $B\ge0$ such that
		$$l_\mathbb{D}(\gamma;[t_1,t_2])\le A\cdot k_\mathbb{D}(\gamma(t_1),\gamma(t_2))+B,$$
		for all $t_2\ge t_1\ge0$.
	\end{defin}
	
	In case there is no possibility of confusion with the constants, we might refer to an $(A,B)$-quasi-geodesic as just a quasi-geodesic in order to simplify the notation. Considering, in addition, the trivial inequality
	$$k_\mathbb{D}(\gamma(t_1),\gamma(t_2))\le l_\mathbb{D}(\gamma;[t_1,t_2]), \quad t_2\ge t_1\ge0,$$
	it is clear that, for a quasi-geodesic, the hyperbolic length among any two points of the curve $\gamma$ is comparable to the hyperbolic distance between these two points. Therefore, a hyperbolic quasi-geodesic presents a behavior that is adequately close to that of the respective hyperbolic geodesic.
	\par Moreover, keeping in mind the conformal invariance of the hyperbolic length and the hyperbolic distance, the definition of a quasi-geodesic extends naturally to any other simply connected domain.
	\par In general, a quasi-geodesic does not need to be a Lipschitz curve, but merely a continuous curve; see \cite{quasi}. However, in the study of semigroups, we usually demand quasi-geodesics to be a priori Lipschitz. For this reason, we demand the Lipschitz condition as well.
	\par Now, we need to somehow connect the notion of a quasi-geodesic with our study on semigroups of holomorphic functions. In \cite{asym}, the authors prove that for a non-elliptic semigroup $(\phi_t)$ with Denjoy-Wolff point $\tau$, every forward trajectory that converges non-tangentially to $\tau$ is actually a quasi-geodesic for the hyperbolic distance $k_\mathbb{D}$ of the unit disk (see also \cite[Theorem 17.3.3]{book}).
	\par A natural question that arises from this fact is whether the same is true for any backward orbit that converges non-tangentially to a point of the unit circle. In fact, the answer is, at least partly for now, positive and this is the main goal of this present article. We need the following definition in order to make a helpful distinction within the class of backward orbits of a semigroup; see \cite[p. 346]{book}.
	\begin{defin}
		Let $(\phi_t)$ be a semigroup in $\mathbb{D}$ and $\gamma:[0,+\infty)\to\mathbb{D}$ a backward orbit. Then, $\gamma$ is called \textit{regular} if
		$$V(\gamma):=\limsup\limits_{t\to+\infty}k_\mathbb{D}(\gamma(t),\gamma(t+1))<+\infty.$$
		The quantity $V(\gamma)$ is called the \textit{hyperbolic step} of $\gamma$.
	\end{defin}
	\par To truly understand the scope of our theorems, it would be useful to mention all the different cases with regard to the convergence of a backward orbit. To this goal, we require some notions concerning semigroups in general. First, a point $\sigma\in\partial\mathbb{D}$ is called a \textit{fixed point} of a semigroup $(\phi_t)$ if $\angle\lim\limits_{z\to\sigma}\phi_t(z)=\sigma$, for all $t\ge0$. Every non-elliptic semigroup has such a fixed point, the Denjoy-Wolff point. Any other fixed point that also has a finite angular derivative $\phi_t^{'}(\sigma)$ is characterized as a \textit{repelling} fixed point of the semigroup. For a repelling fixed point $\sigma$, there exists $\mu\in(-\infty,0)$ such that $\phi_t^{'}(\sigma)=e^{-\mu t}$, for all $t\ge0$. This $\mu$ is called the \textit{repelling spectral value} of $\sigma$. Finally, any fixed point $\sigma$ of the semigroup satisfying $\phi_t^{'}(\sigma)=\infty$, for some (and equivalently all) $t>0$, is characterized as a \textit{super-repelling} fixed point of the semigroup. The Denjoy-Wolff point along with any repelling fixed point of the semigroup are collectively known as the \textit{regular} fixed points of the semigroup. For more information on the fixed points of a semigroup, the reader may refer to \cite[Chapter 12]{book}.
	\par As for the convergence, a regular backward orbit converges non-tangentially to a point of the unit circle if and only if it converges to a repelling fixed point of the semigroup. On the other hand, a non-regular backward orbit converges non-tangentially only if it converges to a super-repelling fixed point, while the converse does not hold. More details with respect to the convergence of a backward orbit follow in Section 2.5. 
	\par We will prove the following:
	\begin{theorem}
		Let $(\phi_t)$ be a non-elliptic semigroup of holomorphic functions in the unit disk. Suppose that $\gamma:[0,+\infty)\to\mathbb{D}$ is a regular backward orbit for $(\phi_t)$ that converges non-tangentially to a point of the unit circle. Then, $\gamma$ is a quasi-geodesic for the hyperbolic distance $k_\mathbb{D}$ of the unit disk.
	\end{theorem}
	
	\par Next, we are going to study backward orbits from a Euclidean perspective. We are in need of some terminology first. Let $\Omega\subsetneq\mathbb{C}$ be a simply connected domain. For $z\in\Omega$, we denote by $\delta_\Omega(z)$ the distance of $z$ from the boundary $\partial\Omega$ (i.e. $\delta_\Omega(z)=\text{dist}(z,\partial\Omega))$. It can be proved (see Section 2) that the image of a backward orbit $\gamma$ through the Koenigs function $h$ is a horizontal half-line. If, in addition, the backward orbit converges non-tangentially to a point of the unit circle, the carrier of the half-line separates $\partial\Omega$, where $\Omega=h(\mathbb{D})$, into two disjoint sets $\partial\Omega^+$ and $\partial\Omega^-$. We denote by $\partial\Omega^+$ the component that is ``over'' the half-line $h\circ\gamma$ and by $\partial\Omega^-$ the one ``below''. Then, for $z\in\Omega$, we set $\delta_\Omega^+(z)=\text{dist}(z,\partial\Omega^+)$ and $\delta_\Omega^-(z)=\text{dist}(z,\partial\Omega^-)$. More information and stricter definitions follow in the next sections. We will prove the following:
	\begin{theorem}
		Let $(\phi_t)$ be a non-elliptic semigroup of holomorphic functions in the unit disk with Koenigs function $h$. Suppose that $\gamma:[0,+\infty)\to\mathbb{D}$ is a backward orbit for $(\phi_t)$ that converges non-tangentially to a point of the unit circle. Set $\Omega=h(\mathbb{D})$. Then, there exists a constant $c\ge1$ such that
		$$\frac{1}{c}\cdot\delta_\Omega^-(h(\gamma(t)))\le\delta_\Omega^+(h(\gamma(t)))\le c\cdot\delta_\Omega^-(h(\gamma(t))),$$
		for all $t\ge0$.
	\end{theorem}
	\par In fact, the above inequality is true for every forward trajectory of a semigroup that converges non-tangentially to the Denjoy-Wolff point (see \cite{asym}). Therefore, it is once again only natural to assume that the same should be true for the respective backward orbits as well.
	\par A first interesting observation into the statement of Theorem 1.2 is the fact that it is not limited only to regular backward orbits. Indeed, the proof that will be provided in Section 4 holds for non-regular backward orbits as well. Theorem 1.2 can be actually expressed in a more general statement that applies to a wider family of curves, but for the purposes of this present article, we will limit ourselves to backward orbits. After the proof, we will give some final insight on backward orbits with respect to the boundary distance.
	\par The structure of the article is as follows: in Section 2, we will provide all the necessary information that will deem useful for the proofs of our theorems. Then, in Section 3, we will deal with our first theorem. Finally, in Section 4, we will present our second theorem along with some relevant corollaries.

	\section{Preliminary Results}
	
	\subsection{Hyperbolic Distance}
	Having already introduced the notion of the hyperbolic distance, we now need to mention some of its properties. First of all, by the definition of the hyperbolic distance with regard to an arbitrary simply connected domain, it becomes clear that the hyperbolic distance is conformally invariant. In addition, the hyperbolic distance has the following monotonicity property: let $\Omega_1, \Omega_2$ be two simply connected domains, other than the complex plane, such that $\Omega_1\subset\Omega_2$. Then, for all $z,w\in\Omega_1$ we have
	$$k_{\Omega_2}(z,w)\le k_{\Omega_1}(z,w).$$
	\subsection{Quasi-hyperbolic Distance}
	Let $\Omega\subsetneq\mathbb{C}$ be a simply connected domain. Then, for any $z_1,z_2\in\Omega$, the \textit{quasi-hyperbolic distance} is defined as
	$$d_\Omega(z_1,z_2)=\inf\limits_{\gamma}\int\limits_{\gamma}\frac{|dz|}{\delta_\Omega(z)},$$
	where the infimum is taken over all curves $\gamma$ that join $z_1$ and $z_2$ inside $\Omega$. The following inequality (\cite[Theorem 5.2.1]{book}) correlates the hyperbolic and quasi-hyperbolic metrics:
	$$\frac{1}{4\delta_\Omega(z)}\le\lambda_\Omega(z)\le\frac{1}{\delta_\Omega(z)},$$
	for all $z\in\Omega$. An immediate corollary of the above inequality is the fact that
	$$\frac{d_\Omega(z_1,z_2)}{4}\le k_\Omega(z_1,z_2)\le d_\Omega(z_1,z_2),$$
	for all $z_1,z_2\in\Omega$. These two inequalities will play a principal role in the proof of our theorems. In particular, this next Distance Lemma provides a useful correlation between the hyperbolic distance and the Euclidean distance from the boundary.
	\begin{lm}{\cite[Theorem 5.3.1]{book}}
		Let $\Omega\subsetneq\mathbb{C}$ be a simply connected domain. Then, for every $z_1,z_2\in\Omega,$
		$$\frac{1}{4}\log\left(1+\frac{|z_1-z_2|}{\min\{\delta_\Omega(z_1),\delta_\Omega(z_2)\}}\right)\le k_\Omega(z_1,z_2)\le\int\limits_\Gamma\frac{|dw|}{\delta_\Omega(w)},$$
		where $\Gamma$ is any piecewise $C^1$-smooth curve joining $z_1$ to $z_2$ inside $\Omega$.
	\end{lm}
	\subsection{Harmonic Measure}
	We provide here some basic facts concerning a second conformal invariant, the harmonic measure. For a comprehensive presentation of its theory, the reader may refer to \cite[Chapter 7]{book}, \cite{garn}, \cite[Chapter 1]{ots}. Let $\Omega\subsetneq\mathbb{C}$ be a simply connected domain with non-polar boundary and $B$ a Borel subset of $\partial\Omega$. Then, the \textit{harmonic measure} of $B$ with respect to $\Omega$ is exactly the solution of the generalized Dirichlet problem for the Laplacian in $\Omega$ with boundary function equal to $1$ on $B$ and to $0$ on $\partial\Omega\setminus B$. For the harmonic measure of $B$ with respect to $\Omega$ and for $z\in\Omega$, we use the notation $\omega(z,B,\Omega)$. It is known that for a fixed $z\in\Omega$, $\omega(z,\cdot,\Omega)$ is a Borel probability measure on $\partial\Omega$. 
	\begin{rem}
		Let $B\subset\partial\mathbb{D}$ be a circular arc with endpoints $\sigma,\tau\in\partial\mathbb{D}$. Then (see e.g. \cite[p.155]{cara}), the level set
		$$L_k=\{z\in\mathbb{D}:\omega(z,B,\mathbb{D})=k\},\quad0<k<1,$$
		is a circular arc (or a diameter in case $B$ is a half-circle and $k=\frac{1}{2}$) with endpoints $\sigma$ and $\tau$ that meets the unit circle with angle $k\cdot\pi$. As a result, a sequence $\{z_n\}\subset\mathbb{D}$ that converges to $\sigma$ or $\tau$, converges non-tangentially if and only if 
		$$0<\liminf\limits_{n\to+\infty}\omega(z_n,B,\mathbb{D})\le\limsup\limits_{n\to+\infty}\omega(z_n,B,\mathbb{D})<1.$$
		If either the liminf is $0$ or the limsup is $1$, then $\{z_n\}$ contains a subsequence that converges tangentially and thus, we cannot say that the convergence of $\{z_n\}$ is non-tangential.
	\end{rem}
	\subsection{Extremal Length}
	Another conformally invariant quantity we are going to need through the course of the article is the extremal length. For a detailed presentation of the theory of the extremal length, the interested reader may refer to \cite{ots}.
	\par Let $\Gamma$ be a family of locally rectifiable curves inside a domain $\Omega\subsetneq\mathbb{C}$. We consider all the non-negative Borel measurable functions $\rho$ defined on $\Omega$ such that the $\rho$\textit{-area}
	$$A(\rho,\Omega)=\iint\limits_{\Omega}\rho(z)^2dxdy,\quad z=x+iy,$$
	is a positive real number. We define the $\rho$\textit{-length} of $\Gamma$ as
	$$L(\rho,\Gamma)=\inf\limits_{\gamma\in\Gamma}\int\limits_{\gamma}\rho(z)|dz|.$$
	Then, the \textit{extremal length} of $\Gamma$ with respect to $\Omega$ is the quantity
	$$\lambda_\Omega(\Gamma)=\sup\limits_{\rho}\frac{L(\rho,\Gamma)^2}{A(\rho,\Omega)},$$
	where the supremum is taken over all non-negative Borel measurable functions $\rho$ such that $A(\rho,\Omega)\in(0,+\infty)$.
	\par Next, we consider two families of locally rectifiable curves $\Gamma_1,\Gamma_2$ defined on two domains $\Omega_1,\Omega_2\subsetneq\mathbb{C}$ respectively. If the family $\Gamma_1$ is a subfamily of the family $\Gamma_2$, then by the definition of extremal length, it is easily understood that $\lambda_{\Omega_1}(\Gamma_1)\ge\lambda_{\Omega_2}(\Gamma_2)$. Moreover, we are going to need the following lemma.
	\begin{lm}{\cite[Lemma 8.4]{vuor}}
		If every curve of the family $\Gamma_1$ contains a subcurve that belongs to the family $\Gamma_2$, then $\lambda_{\Omega_1}(\Gamma_1)\ge\lambda_{\Omega_2}(\Gamma_2)$.
	\end{lm}
	\par In this present article, along with the above two properties, we will mostly need one specific instance of the extremal length. Once again, let $\Omega\subsetneq\mathbb{C}$ be a domain and let $E,F\subset\overline{\Omega}$. Then, the \textit{extremal distance} between $E$ and $F$ inside $\Omega$ is defined as the quantity
	$$\lambda(E,F,\Omega)=\lambda_\Omega(\Gamma),$$
	where $\Gamma$ is the family of all locally rectifiable curves that join $E$ to $F$ inside $\Omega$. By extension, the extremal distance is also conformally invariant.
	\par A very useful tool that connects the extremal distance with the harmonic measure is the following inequality, the so-called \textit{Beurling's estimate}.
	\begin{lm}{\cite[Theorem H.7]{garn}}
		Let $\Omega$ be a simply connected domain and let $B$ be a finite union of arcs lying on $\partial\Omega$. Let $z\in\Omega$. Then,
		$$\omega(z,B,\Omega)\le C\cdot\exp(-\pi\cdot\lambda(\gamma,B,\Omega)),$$
		for every Jordan arc $\gamma$ joining $z$ to $\partial\Omega\setminus B$ inside $\Omega$, where $C>0$ is an absolute constant.
	\end{lm}
	\par Finally, for $r\in(0,1)$, we call the doubly connected domain $\mathbb{D}\setminus[0,r]$ a \textit{Gr\"{o}tzsch ring} (see Figure 1). Through the study of Gr\"{o}tzsch rings (see \cite{ots}), a lot of information has arisen about the extremal distance between the line segment $[0,r]$ and the unit circle inside the unit disk. In particular, it has been proved that $\lim\limits_{r\to0^+}\lambda([0,r],\partial\mathbb{D},\mathbb{D})=+\infty$ (see \cite[Theorem 2.80]{ots}), a fact that we will need in the final stages of the proof of Theorem 1.2.
	\begin{figure}
		\centering
		\includegraphics[scale=0.6]{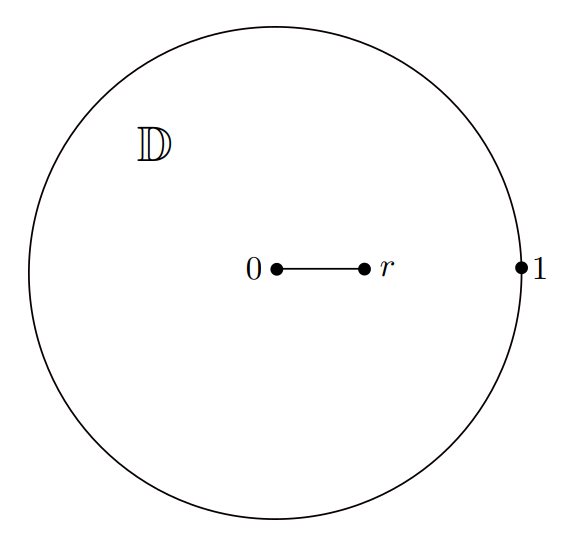}
		\caption{A Gr\"{o}tzsch ring}
	\end{figure}
	\subsection{Backward Orbits}
	Let $\gamma:[0,+\infty)\to\mathbb{D}$ be a backward orbit for a non-elliptic semigroup $(\phi_t)$ with Denjoy-Wolff point $\tau$ and Koenigs function $h$. Set $\Omega=h(\mathbb{D})$. By the definition of backward orbits, we have
	$$\phi_t(\gamma(t))=\gamma(t-t)=\gamma(0),$$
	for all $t\ge0$. Remembering the basic property of the Koenigs function, we get $h(\gamma(t))=h(\gamma(0))-t, t\ge0$. Therefore, we see that $h$ linearizes backward orbits, turning them to half-lines that converge to $\infty$ through the negative direction (i.e. with strictly decreasing real part and constant imaginary part).
	\par Next, since $\gamma(0)\in\mathbb{D}$, we can find the forward trajectory of $\gamma(0)$ under the semigroup along with its image through $h$. More precisely, we can find the image of the set $\{\phi_t(\gamma(0)):t\ge0\}$. In a similar fashion, its image in $\Omega$ is the half-line $\{h(\gamma(0))+t:t\ge0\}$ which converges to $\infty$ through the positive direction. As a result, the union of the images of the backward orbit and the forward trajectory that emanate from $\gamma(0)$ is the horizontal line $\{h(\gamma(0))+t:t\in\mathbb{R}\}$. Clearly, this line separates $\partial\Omega$ into two disjoint sets $\partial\Omega^+=\{w\in\partial\Omega:\Im w>\Im h(\gamma(0))\}$ and $\partial\Omega^-=\{w\in\partial\Omega:\Im w<\Im h(\gamma(0))\}$.
	\par Having already mentioned some details in the Introduction, we need some information concerning the convergence of a backward orbit. All the following results are drawn from \cite[Chapter 13]{book}. If $\gamma$ is regular, then either $\lim\limits_{t\to+\infty}\gamma(t)=\sigma\in\partial\mathbb{D}\setminus\{\tau\}$, where $\sigma$ is a repelling fixed point of the semigroup, and the convergence is non-tangential or $\lim\limits_{t\to+\infty}\gamma(t)=\tau$ and the convergence is tangential. In the former case, the associated planar domain $\Omega$ must contain a maximal horizontal strip and the image of the backward orbit through $h$ is contained in this strip. In the latter case, $\Omega$ must contain a maximal horizontal half-plane and the image of the backward orbit through $h$ is contained in this half-plane.
	\par Here, it seems useful to make some comments concerning prime ends. Since $h$ maps the unit disk $\mathbb{D}$ conformally onto the Koenigs domain $\Omega$, by Carath\'{e}dory's Theorem (\cite[Theorem 4.2.3]{book}), the points of the unit circle $\partial\mathbb{D}$ correspond one-to-one to the prime ends of $\Omega$. However, multiple prime ends of $\Omega$ might correspond to the same point of $\partial\Omega$. For example, if for some backward orbit $\gamma:[0,+\infty)\to\mathbb{D}$ both $\partial\Omega^+$ and $\partial\Omega^-$ are non-empty, then there are at least two prime ends of $\Omega$ corresponding to $\infty$. This is because $\lim\limits_{t\to+\infty}h(\gamma(t))=\lim\limits_{t\to+\infty}h(\phi_t(\gamma(0)))=\infty$, while $\lim\limits_{t\to+\infty}\gamma(t)\neq\tau$. Combining all the above, we can say that $\partial\Omega^+$ corresponds to a set of prime ends of $\Omega$ which in turn corresponds through $h^{-1}$ to a set of points $\partial\mathbb{D}^+$ on the unit circle. Similarly, $\partial\Omega^-$ corresponds to a set $\partial\mathbb{D}^-$. It is easy to check that $\partial\mathbb{D}^+$ and $\partial\mathbb{D}^-$ are disjoint circular arcs with endpoints $\lim\limits_{t\to+\infty}\gamma(t)$ and $\tau$, whose union is the whole unit circle. For more information on prime ends and how they connect to semigroups of holomorphic functions, we refer to \cite[Chapter 11]{book}.
	\par Using a previous remark that the image of a backward orbit in $\Omega$ must tend to $\infty$ through the negative direction, we can understand that there are only three categories left for the image through $h$ of a non-regular backward orbit (see Figure 2):
	\begin{figure}
			\subfloat[$\gamma_1$ is non-regular, $\gamma_2$ is regular and they both converge to a repelling point]{\label{b}\includegraphics[scale=0.45]{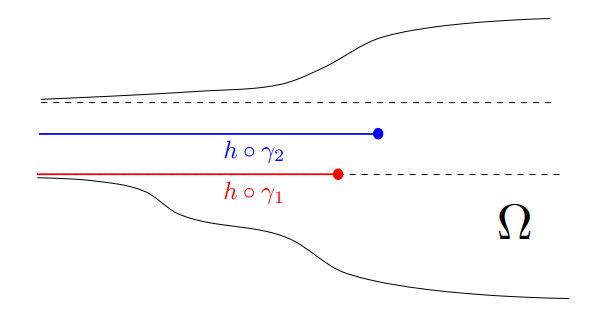}}\hfill
			\subfloat[$\gamma_1$ is non-regular, $\gamma_2$ is regular and they both converge to the Denjoy-Wolff point]{\label{a}\includegraphics[scale=0.45]{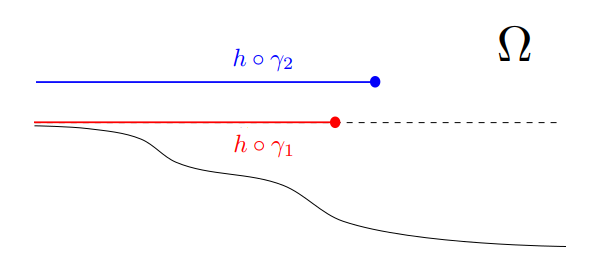}}	\par
			\subfloat[$\gamma$ is non-regular and converges to a super-repelling point]{\label{c}\includegraphics[scale=0.45]{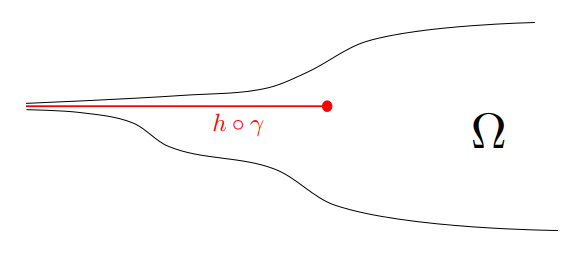}}
			\caption{Images of backward orbits through the Koenigs function}
	\end{figure}
	\begin{enumerate}[(i)]
		\item either it is part of the boundary of a maximal horizontal strip, in which case the non-regular backward orbit itself converges tangentially to a repelling fixed point,
		\item or it is part of the boundary of a maximal horizontal half-plane, in which case the non-regular backward orbit itself converges tangentially to the Denjoy-Wolff point,
		\item or the image gets arbitrarily close to both $\partial\Omega^+$ and $\partial\Omega^-$, in which case the non-regular backward orbit itself converges to a super-repelling fixed point either tangentially or non-tangentially, or can even contain sequences of points that converge tangentially and other sequences that converge non-tangentially.
	\end{enumerate}
	It can be easily proved that any backward orbit that converges tangentially to a point of the unit circle cannot be a quasi-geodesic for the hyperbolic distance $k_\mathbb{D}$ of the unit disk. Therefore, we limit our search for quasi-geodesics, only in the case of a backward orbit converging non-tangentially.
	\subsection{Infinitesimal Generators}
	Let $(\phi_t)$ be a semigroup in $\mathbb{D}$. Then, there exists a unique holomorphic function $G:\mathbb{D}\to\mathbb{C}$ such that
	$$\frac{\partial\phi_t(z)}{\partial t}=G(\phi_t(z)),\quad\quad z\in\mathbb{D},\quad t\ge0.$$
	The function $G$ is called the \textit{infinitesimal generator} of $(\phi_t)$; see \cite[Chapter 10]{book}. We need the notion of the infinitesimal generators in order to prove that certain backward orbits are actually Lipschitz curves. We are going to mainly need the following lemma:
	\begin{lm}{\cite[Proposition 12.2.2]{book}}
		Let $G$ be the infinitesimal generator of a semigroup in $\mathbb{D}$ and let $\sigma\in\partial\mathbb{D}$ be a repelling fixed point of the semigroup. Then, $\angle\lim\limits_{z\to\sigma}G(z)=0$.
	\end{lm}

	\section{Backward Orbits and Quasi-Geodesics}
	The first goal of this article is to prove that any regular backward orbit that converges non-tangentially to a point of the unit circle is a quasi-geodesic for the hyperbolic distance of the unit disk. By the definition of quasi-geodesics, we first need to show that any such backward orbit is a Lipschitz curve. The proof of the following proposition follows almost exactly the same steps as the proof of Proposition 10.1.7(4) in \cite{book}.
	\begin{prop}
		Let $(\phi_t)$ be a non-elliptic semigroup in $\mathbb{D}$ and let $\gamma:[0,+\infty)\to\mathbb{D}$ be a regular backward orbit for $(\phi_t)$ that converges to a point of the unit circle non-tangentially. Then $\gamma$ is a Lipschitz curve.
	\end{prop}
	\begin{proof}
		Let $\lim\limits_{t\to+\infty}\gamma(t)=\sigma$. By Section 2.5, the fact that $\gamma$ converges non-tangentially to $\sigma$, forces $\sigma$ to be a repelling fixed point of $(\phi_t)$. Let $s\ge t_2\ge t_1\ge0$ and set $s_1=s-t_1$ and $s_2=s-t_2$. Obviously, $s_1\ge s_2$. Then,
		$$|\gamma(t_2)-\gamma(t_1)|=|\gamma(s-s_2)-\gamma(s-s_1)|=|\phi_{s_2}(\gamma(s))-\phi_{s_1}(\gamma(s))|,$$
		where the last equality follows from the definition of a backward orbit. Let $G$ be the infinitesimal generator of $(\phi_t)$. Then,
		\begin{eqnarray*}
		|\gamma(t_2)-\gamma(t_1)|&=&|\phi_{s_1}(\gamma(s))-\phi_{s_2}(\gamma(s))|\\
		&\le&\int\limits_{s_2}^{s_1}\left|\frac{\partial\phi_t(\gamma(s))}{\partial t}\right|dt\\
		&=&\int\limits_{s_2}^{s_1}|G(\phi_t(\gamma(s)))|dt\\
		&=&\int\limits_{s_2}^{s_1}|G(\gamma(s-t))|dt.
		\end{eqnarray*}
		However, $\lim\limits_{t\to+\infty}\gamma(t)=\sigma$ and $\angle\lim\limits_{z\to\sigma}G(z)=0$. As a consequence, $G$ is bounded on every Stolz angle with vertex $\sigma$. Since $\gamma$ converges to $\sigma$ non-tangentially, there must exist a Stolz angle $S$ with vertex $\sigma$ such that $\gamma([0,+\infty))\subset S$. Therefore, there exists an absolute constant $C>0$, depending only on $S$, such that $|G(\gamma(t))|<C$, for every $t\ge0$. Combining with the above relations, we get $|\gamma(t_2)-\gamma(t_1)|\le C(s_1-s_2)=C(t_2-t_1)$. This last relation yields the desired result.
	\end{proof}
	Now, we are ready for the proof of Theorem 1.1.
	
	\textit{Proof of Theorem 1.1} First of all, since $\gamma(0)\in\mathbb{D}$ and $\gamma$ converges to a point of the unit circle, it is trivial that $\lim\limits_{t\to+\infty}k_\mathbb{D}(\gamma(0),\gamma(t))=+\infty$. As a result, $\gamma$ satisfies the elementary necessary condition to be considered as a candidate for a quasi-geodesic. Moreover, as we showed before, if $\gamma$ converged to the Denjoy-Wolff point $\tau$, the convergence would be tangential. Therefore, $\gamma$ converges to a point $\sigma\in\partial\mathbb{D}\setminus\{\tau\}$ and $\sigma$ must be a repelling fixed point of $(\phi_t)$, because $\gamma$ is regular. For the remainder of the proof, we denote by $\mu\in(-\infty,0)$ the repelling spectral value of $\sigma$.
	\par Set $\Omega=h(\mathbb{D})$, where $h$ is the Koenigs function of $(\phi_t)$. According to \cite{str}, $\Omega$ must contain a maximal horizontal strip of amplitude $-\frac{\pi}{\mu}$. We denote by $S$ this maximal strip. Then, we can write $S=\{z\in\mathbb{C}:|\Im z-a^*|<-\frac{\pi}{2\mu}\}$, where $a^*=\lim\limits_{r\to1^-}\Im h(r\sigma)$.
	We know that the set $\Delta=h^{-1}(S)$ is called a petal of $(\phi_t)$ and every regular backward orbit converging to $\sigma$ is contained inside $\Delta$. Therefore, $\gamma([0,+\infty))\subset\Delta$ or equivalently $h\circ\gamma([0,+\infty))\subset S$. In addition, since $\gamma$ is a backward orbit, it is true that
	$$\phi_t(\gamma(t))=\gamma(t-t)=\gamma(0),$$
	for all $t\ge0$. Turning to $\Omega$ through $h$, $h(\phi_t(\gamma(t)))=h(\gamma(t))+t$, which yields
	$$h(\gamma(t))=h(\gamma(0))-t,$$
	for all $t\ge0$. On the other hand, a similar computation shows that the image of the forward trajectory of $\gamma(0)$ through $h$ is the horizontal half-line $\{h(\gamma(0))+t:t\ge0\}$. Thus, the union of the images of the backward and forward trajectories of $\gamma(0)$ through $h$ is a horizontal line which separates $\partial\Omega$ into two disjoint sets. We write $\partial\Omega^+=\{w\in\partial\Omega:\Im w>\Im h(\gamma(0))\}$ and $\partial\Omega^{-}=\{w\in\partial\Omega:\Im w<\Im h(\gamma(0))\}$. It is easy to see that $\partial\Omega^+\cap\partial\Omega^-=\emptyset$ and $\partial\Omega^+\cup\partial\Omega^-=\partial\Omega$. For $z\in\Omega$, we have already set $\delta_\Omega(z)=\text{dist}(z,\partial\Omega)$. For the remainder of the article, we also set $\delta_\Omega^+(z)=\text{dist}(z,\partial\Omega^+)$ and $\delta_\Omega^-(z)=\text{dist}(z,\partial\Omega^-)$.
	\par We find $p^+=x^++iy^+\in\partial\Omega^+$ such that $\delta_\Omega^+(h(\gamma(0)))=|h(\gamma(0))-p^+|$. Likewise, we find $p^-=x^-+iy^-\in\partial\Omega^-$ such that $\delta_\Omega^-(h(\gamma(0)))=|h(\gamma(0))-p^-|$. Setting $x=\min\{x^+,x^-\}$ and by the convexity in the positive direction of $\Omega$, we get that the horizontal half-lines $L^+=\{w:\Re w\le x,\Im w=y^+\}$ and $L^-=\{w:\Re w\le x, \Im w=y^-\}$ are both subsets of $\mathbb{C}\setminus\Omega$. Therefore, $S\subset\Omega\subset B:=\mathbb{C}\setminus(L^+\cup L^-)$ (see Figure 3). Moreover, since $\lim\limits_{t\to+\infty}\Re h(\gamma(t))=-\infty$, there exists a $t_0\ge0$ such that $\Re h(\gamma(t))\le x$, for all $t\ge t_0$.
	\begin{figure}
		\centering
		\includegraphics[scale=0.8]{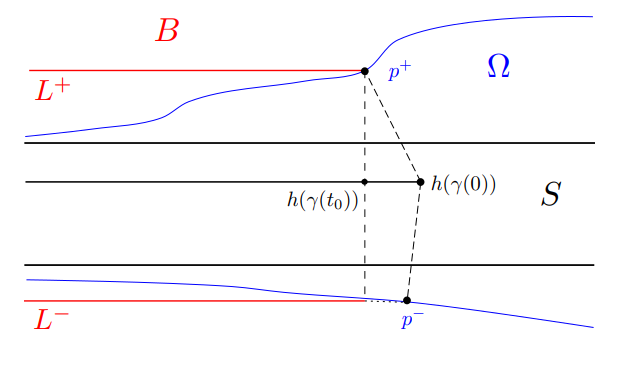}
		\caption{The sets of Theorem 1.1}
	\end{figure}
	\par Given the geometrical properties of $S$ and $B$ and the fact that the set $h(\gamma([0,+\infty)))$ is a horizontal half-line, there exist positive constants $c,d$ such that $\delta_S(h(\gamma(t)))=c$ and $\delta_B(h(\gamma(t)))=d$, for all $t\ge t_0$. Combining all the above and making use of the monotonicity and the conformal invariance of the hyperbolic metric, we have for $t_2\ge t_1\ge t_0$:
	\begin{eqnarray*}
		l_\mathbb{D}(\gamma;[t_1,t_2])=l_\Omega(h\circ\gamma;[t_1,t_2])&\le& l_S(h\circ\gamma;[t_1,t_2]) \\
		&=&\int\limits_{t_1}^{t_2}\lambda_S(h(\gamma(t)))|(h\circ\gamma)'(t)|dt \\
		&\le&\int\limits_{t_1}^{t_2}\frac{|(h\circ\gamma)'(t)|dt}{\delta_S(h(\gamma(t)))}\\
		&=&\frac{4d}{c}\int\limits_{t_1}^{t_2}\frac{|(h\circ\gamma)'(t)|dt}{4\delta_B(h(\gamma(t)))}.
	\end{eqnarray*} 
	\par The horizontal line $\{w:\Im w=\frac{y^++y^-}{2}\}$ is an axis of symmetry for $B$. Therefore, the curve $\eta:[t_0,+\infty)\to B$ with $\eta(t)=\Re h(\gamma(t))+i\frac{y^++y^-}{2}$ is a geodesic arc for the hyperbolic distance $k_B$ of $B$. Setting $\epsilon=\frac{y^+-y^-}{2}$, it is evident that $\delta_B(\eta(t))=\epsilon$, for all $t\ge t_0$. Furthermore, an easy direct computation shows that $|(h\circ\gamma)'(t)|=|\eta'(t)|=1$, for all $t\ge t_0$. Continuing the above relations, we get
	\begin{eqnarray*}
		l_\mathbb{D}(\gamma;[t_1,t_2])&\le&\frac{4\epsilon}{c}\int\limits_{t_1}^{t_2}\frac{|(h\circ\gamma)'(t)|dt}{4\delta_B(\eta(t))}\\
		&=&\frac{4\epsilon}{c}\int\limits_{t_1}^{t_2}\frac{|\eta'(t)|dt}{4\delta_B(\eta(t))}\\
		&\le&\frac{4\epsilon}{c}l_B(\eta;[t_1,t_2])\\
		&=&\frac{4\epsilon}{c}k_B(\eta(t_1),\eta(t_2)),
	\end{eqnarray*}
	for all $t_2\ge t_1\ge t_0$, where the last equality is due to the fact that $\eta$ is a geodesic arc. However, by the triangle inequality, we get 
	$$k_B(\eta(t_1),\eta(t_2))\le k_B(h(\gamma(t_1)),h(\gamma(t_2)))+\sum\limits_{i=1}^{2}k_B(h(\gamma(t_i)),\eta(t_i)).$$
	\par For $t\ge t_0$, we denote by $\Gamma_t$ the vertical line segment joining $\eta(t)$ to $h(\gamma(t))$. We can parametrize it as follows: $\Gamma_t:[0,1]\to B$, with $\Gamma_t(s)=(1-s)\eta(t)+s\cdot h(\gamma(t))$. Then, by Lemma 2.1,
	$$k_B(\eta(t),h(\gamma(t)))\le\int\limits_{0}^{1}\frac{|\Gamma_t'(s)|ds}{\delta_B(\Gamma_t(s))}\le\int\limits_{0}^{1}\frac{|h(\gamma(t))-\eta(t)|ds}{\min\limits_{s\in[0,1]}\delta_B(\Gamma_t(s))}\le\int\limits_{0}^{1}\frac{\epsilon ds}{d}=\frac{\epsilon}{d}.$$
	Combining this last result with the triangle inequality and returning to the previous relations, we find that
	$$l_\mathbb{D}(\gamma;[t_1,t_2])\le\frac{4\epsilon}{c}k_B(h(\gamma(t_1)),h(\gamma(t_2)))+\frac{8\epsilon^2}{cd},$$
	for all $t_2\ge t_1\ge t_0$. But $\Omega\subset B$, so 
	$$k_B(h(\gamma(t_1)),h(\gamma(t_2)))\le k_\Omega(h(\gamma(t_1)),h(\gamma(t_2)))=k_\mathbb{D}(\gamma(t_1),\gamma(t_2)).$$
	Setting $A=\frac{4\epsilon}{c}$ and $B=\frac{8\epsilon^2}{cd}+l_\mathbb{D}(\gamma;[0,t_0])$, we finally get
	$$l_\mathbb{D}(\gamma;[t_1,t_2])\le A\cdot k_\mathbb{D}(\gamma(t_1),\gamma(t_2))+B,$$
	for all $t_2\ge t_1\ge0$, and this yields the desired result. \qed
	
	\begin{rem}
		Perhaps the condition of regularity for Theorem 1.1 is actually not necessary. Indeed, even if $\gamma$ is non-regular, it can still be a quasi-geodesic and even a geodesic. This happens, for instance, when $\Omega$ contains no horizontal strips and is symmetric with respect to a horizontal line and $h(\gamma([0,+\infty)))$ lies on this line. So, a natural conjecture is that Theorem 1.1 is true for non-regular backward orbits that converge non-tangentially as well.
	\end{rem}
	
	\section{Backward Orbits and Boundary Distance}
	In this section, we will give a property of backward orbits in terms of Euclidean geometry. We will first prove Theorem 1.2 and in the end we will state some corollaries.\\
	
	\textit{Proof of Theorem 1.2} We aim for a contradiction. To this goal, suppose that $\limsup\limits_{t\to+\infty}\frac{\delta_\Omega^+(h(\gamma(t)))}{\delta_\Omega^-(h(\gamma(t))}=+\infty$. Then, there exists a strictly increasing sequence $\{t_n\}\subset[0,+\infty)$ with $\lim\limits_{n\to+\infty}t_n=+\infty$ such that $\lim\limits_{n\to+\infty}\frac{\delta_\Omega^+(h(\gamma(t_n)))}{\delta_\Omega^-(h(\gamma(t_n)))}=+\infty$. For the sake of convenience in notation, we set $\delta^+(n)=\delta_\Omega^+(h(\gamma(t_n)))$ and $\delta^-(n)=\delta_\Omega^-(h(\gamma(t_n)))$. Without loss of generality, by extracting a subsequence, we can assume that $\delta^+(n)>\delta^-(n)$, for all $n\in\mathbb{N}$. In addition, we denote by $p_n^+$ a point of $\partial\Omega^+$ such that $\delta^+(n)=|h(\gamma(t_n))-p_n^+|$ and similarly, by $p_n^-$ a point of $\partial\Omega^-$ such that $\delta^-(n)=|h(\gamma(t_n))-p_n^-|$. Moreover, let $\gamma_n^+$ be the line segment $[h(\gamma(t_n)),p_n^+)$ and $\gamma_n^-$ the line segment $[h(\gamma(t_n)),p_n^-)$. Clearly, $\gamma_n^+,\gamma_n^-\subset\Omega$, for all $n\in\mathbb{N}$ (see Figure 4). Finally, we denote by $\Gamma_n^-$ the family of all Jordan arcs that join $h(\gamma(t_n))$ to $\partial\Omega^-$ inside $\Omega$. Obviously, $\gamma_n^-\in\Gamma_n^-$, for all $n\in\mathbb{N}$.
	\begin{figure}
		\centering
		\includegraphics[scale=0.85]{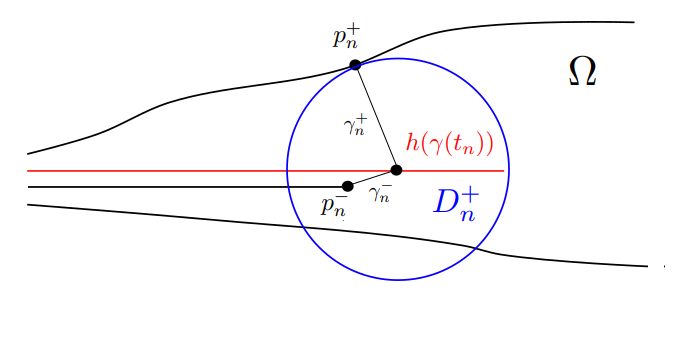}
		\caption{The sets of Theorem 1.2}
	\end{figure}
	\par We need to return, for a bit, to the unit disk $\mathbb{D}$. The line $\{w:\Im w=\Im h(\gamma(0))\}$ separates $\partial\Omega$ into $\partial\Omega^+$ and $\partial\Omega^-$. We know that this line is the image through $h$ of the set $\gamma([0,+\infty))\cup\{\phi_t(\gamma(0)):t\ge0\}$, where $\{\phi_t(\gamma(0)):t\ge0\}$ is the forward trajectory of $\gamma(0)$ under the semigroup $(\phi_t)$. The above union is a maximal invariant curve of $(\phi_t)$ (see \cite[Section 13.3]{book}) and separates the unit circle $\partial\mathbb{D}$ into two circular arcs $\partial\mathbb{D}^+$ and $\partial\mathbb{D}^-$ with endpoints the Denjoy-Wolff point $\tau$ and the point $\sigma$ where $\gamma$ converges. Let's say that $\partial\mathbb{D}^+$ corresponds through $h$ to $\partial\Omega^+$ and $\partial\mathbb{D}^-$ to $\partial\Omega^-$ respectively. Then, by the bijective nature of the Koenigs function $h$, the family of all Jordan arcs that join $\gamma(t_n)$ to $\partial\mathbb{D}^-$ inside $\mathbb{D}$ certainly contains the family $h^{-1}(\Gamma_n^-)$. \par For $n\in\mathbb{N}$, Beurling's estimate in Lemma 2.3 yields
	$$\omega(\gamma(t_n),\partial\mathbb{D}^+,\mathbb{D})\le C\cdot\exp(-\pi\cdot\lambda(\xi,\partial\mathbb{D}^+,\mathbb{D})),$$
	for all Jordan arcs $\xi\in h^{-1}(\Gamma_n^-)$.
	Using the conformal invariance of both the harmonic measure and the extremal distance, we get
	$$\omega(h(\gamma(t_n)),\partial\Omega^+,\Omega)\le C\cdot\exp(-\pi\cdot\lambda(\xi',\partial\Omega^+,\Omega)),$$ 
	for all Jordan arcs $\xi'\in\Gamma_n^-$.
	Since $\gamma_n^-\in\Gamma_n^-$, taking limits as $n\to+\infty$, we have
	\begin{equation}
	\limsup\limits_{n\to+\infty}\omega(h(\gamma(t_n)),\partial\Omega^+,\Omega)\le C\cdot\exp(-\pi\cdot\liminf\limits_{n\to+\infty}\lambda(\gamma_n^-,\partial\Omega^+,\Omega)).
	\end{equation}
	\par We set $D_n^-=D(h(\gamma(t_n)),\delta^-(n))$, where by $D(z,r)$, $z\in\mathbb{C},r>0$, we denote the Euclidean disk of center $z$ and radius $r$. Then, the line segment $\gamma_n^-$ is a radius of $D_n^-$. In a similar manner, we consider the disks $D_n^+=D(h(\gamma(t_n)),\delta^+(n))$, each one having the line segment $\gamma_n^+$ as its radius. Bearing in mind the assumption $\delta^+(n)>\delta^-(n)$, by construction, $D_n^-\subset\Omega$, for all $n\in\mathbb{N}$, while the same is not true for the disks $D_n^+$. Nevertheless, it is evident that $\gamma_n^-\subset D_n^+$, for all $n\in\mathbb{N}$ (see Figure 4).
	\par We denote by $\Omega_n$ the connected component of $\Omega\cap D_n^+$ that contains $\gamma_n^-$ and set $E_n=\partial\Omega_n\setminus\partial\Omega^-$. Every curve that joins $\gamma_n^-$ to $\partial\Omega^+$ inside $\Omega$ contains a subcurve joining $\gamma_n^-$ to $E_n$ inside $\Omega_n$. By Lemma 2.2,
	$$\lambda(\gamma_n^-,\partial\Omega^+,\Omega)\ge\lambda(\gamma_n^-,E_n,\Omega_n),$$
	for every $n\in\mathbb{N}$. Moreover, the family of curves joining $\gamma_n^-$ to $E_n$ inside $\Omega_n$ is a subfamily of the family of curves joining $\gamma_n^-$ to $\partial D_n^+$ inside $D_n^+$. Therefore,
	$$\lambda(\gamma_n^-,E_n,\Omega_n)\ge\lambda(\gamma_n^-,\partial D_n^+,D_n^+),$$
	for every $n\in\mathbb{N}$.
	 Combining, we get
	\begin{equation}
		\lambda(\gamma_n^-,\partial\Omega^+,\Omega)\ge\lambda(\gamma_n^-,\partial D_n^+,D_n^+),
	\end{equation}
	for all $n\in\mathbb{N}$. By the conformal invariance of the extremal distance, we have
	\begin{eqnarray*}
		\lambda(\gamma_n^-,\partial D_n^+, D_n^+)&=&\lambda(\gamma_n^--h(\gamma(t_n)),\partial D_n^+-h(\gamma(t_n)), D(0,\delta^+(n)))\\
		&=&\lambda\left(\frac{\gamma_n^--h(\gamma(t_n))}{\delta^+(n)},\partial\mathbb{D},\mathbb{D}\right).
	\end{eqnarray*}
	The line segment $\gamma_n^--h(\gamma(t_n))$ starts at $0$, while its length is $\delta^-(n)$, so by a suitable rotation we can map the line segment $\frac{\gamma_n^--h(\gamma(t_n))}{\delta^+(n)}$ conformally onto the line segment $\left[0,\frac{\delta^-(n)}{\delta^+(n)}\right).$ Therefore, through this chain of conformal mappings, we get
	$$\lambda(\gamma_n^-,\partial D_n^+, D_n^+)=\lambda\left(\left[0,\frac{\delta^-(n)}{\delta^+(n)}\right],\partial\mathbb{D},\mathbb{D}\right).$$ 
	By our hypothesis, $\lim\limits_{n\to+\infty}\frac{\delta^-(n)}{\delta^+(n)}=0$. Through the theory of Gr\"{o}tzsch domains, this implies that $\lim\limits_{n\to+\infty}\lambda(\gamma_n^-,\partial D_n^+, D_n^+)=+\infty$. Returning to (2), this yields $\lim\limits_{n\to+\infty}\lambda(\gamma_n^-,\partial\Omega^+,\Omega)=+\infty$. Finally, going back to (1) and Beurling's estimate, we find that $\lim\limits_{n\to+\infty}\omega(h(\gamma(t_n)),\partial\Omega^+,\Omega)=0$. By Remark 2.1, this signifies that the sequence $\{\gamma(t_n)\}$ converges to $\sigma$ tangentially, as $n\to+\infty$. Contradiction!
	\par Consequently, $\limsup\limits_{t\to+\infty}\frac{\delta_\Omega^+(h(\gamma(t)))}{\delta_\Omega^-(h(\gamma(t)))}<+\infty$. Following a similar procedure, we can prove that $\liminf\limits_{t\to+\infty}\frac{\delta_\Omega^+(h(\gamma(t)))}{\delta_\Omega^-(h(\gamma(t)))}>0$. The combination of these last two relations certifies that there exists indeed an absolute constant $c\ge1$ such that
	$$\frac{1}{c}\cdot\delta_\Omega^-(h(\gamma(t)))\le\delta_\Omega^+(h(\gamma(t)))\le c\cdot\delta_\Omega^-(h(\gamma(t))),$$
	for all $t\ge0$. Thus, Theorem 1.2 is proved. \qed
	
	\begin{rem}
		As we can see from the proof, the curve $\gamma$ needed to just be a backward orbit. There was no added condition of regularity, contrary to Theorem 1.1. Furthermore, the proof itself shows that if either $\liminf\limits_{t\to+\infty}\frac{\delta_\Omega^+(h(\gamma(t)))}{\delta_\Omega^-(h(\gamma(t)))}=0$ or $\limsup\limits_{t\to+\infty}\frac{\delta_\Omega^+(h(\gamma(t)))}{\delta_\Omega^-(h(\gamma(t)))}=+\infty$ (or both), then $\gamma$ converges to $\sigma$ by a set of angles that contains either $\{0\}$ or $\{\pi\}$ (or both).
	\end{rem}
	The preceding remark leads to a partial converse of Theorem 1.2. Before that, we need to return to boundary distances again. In some cases, the image of a backward orbit $\gamma$ through the Koenigs function $h$ might not separate the boundary of $\Omega=h(\mathbb{D})$ into two components. For example, if $\gamma$ is a regular backward orbit that converges to the Denjoy-Wolff point $\tau$ of a semigroup $(\phi_t)$, then $\Omega$ contains a maximal horizontal half-plane $H$ and $h(\gamma([0,+\infty)))\subset H$. Suppose that the half-plane is of the form $H=\{w\in\mathbb{C}:\Im w>a\}$, for some $a\in\mathbb{R}$ (see Figure 2). Then, maintaining the previous symbolisms, $\partial\Omega^+=\emptyset$, while $\lim\limits_{t\to+\infty}\delta_\Omega^-(h(\gamma(t)))=\Im h(\gamma(0))-a$ due to the maximality of $H$. In such a case, we can write $\delta_\Omega^+(h(\gamma(t)))=+\infty$, for all $t\in[0+\infty)$. With this in mind, we have the following corollary.
	\begin{cor}
		Let $(\phi_t)$ be a non-elliptic semigroup in $\mathbb{D}$ with Koenigs function $h$. Suppose that $\gamma:[0,+\infty)\to\mathbb{D}$ is a backward orbit for $(\phi_t)$ that converges tangentially to a regular fixed point of the semigroup. Then, the limit  $\lim\limits_{t\to+\infty}\frac{\delta_\Omega^+(h(\gamma(t)))}{\delta_\Omega^-(h(\gamma(t)))}$ is equal to 0 or $+\infty$.
	\end{cor}
	\begin{proof}
		 The type of convergence in the hypothesis can happen in three distinct cases, as we have already seen in Section 2.5: either $\gamma$ is regular and converges to the Denjoy-Wolff point, or $\gamma$ is non-regular and converges to the Denjoy-Wolff point, or finally $\gamma$ is non-regular and converges to a repelling fixed point of $(\phi_t)$. In the first case, $h\circ\gamma([0,+\infty))$ is contained in a maximal horizontal half-plane contained in $\Omega=h(\mathbb{D})$. In the second case, $h\circ\gamma([0,+\infty))$ lies on the boundary of such a half-plane. Finally, in the third case, $h\circ\gamma([0,+\infty))$ lies on the boundary of a maximal horizontal strip contained in $\Omega$. It is easy to see that in any case, the desired result is true.
		
	\end{proof}
	Studying again the proof of Theorem 1.2, we can also make a conclusion about a non-regular backward orbit $\gamma:[0,+\infty)\to\mathbb{D}$ that lands at a super-repelling fixed point $\sigma$. In this case, as we mentioned before, the convergence can be of any kind. More specifically, the cluster set of $\arg(1-\bar{\sigma}\gamma(t))$ can be any compact and connected subset of $[-\frac{\pi}{2},\frac{\pi}{2}]$ (see \cite{kelg}). Keeping in mind that conformal mappings preserve orientations and angles, it is directly deduced that:
	\begin{enumerate}[(i)]
	\item $\liminf\limits_{t\to+\infty}\frac{\delta_\Omega^+(h(\gamma(t)))}{\delta_\Omega^-(h(\gamma(t)))}=0$ if and only if $\limsup\limits_{t\to+\infty}\arg(1-\bar{\sigma}\gamma(t))=\frac{\pi}{2}$ (i.e. there exists a strictly increasing sequence $\{t_n\}\subset[0,+\infty)$ with $\lim\limits_{n\to+\infty}t_n=+\infty$ such that $\{\gamma(t_n)\}$ converges to $\sigma$ by angle $0$),
	\item $\limsup\limits_{t\to+\infty}\frac{\delta_\Omega^+(h(\gamma(t)))}{\delta_\Omega^-(h(\gamma(t)))}=+\infty$ if and only if $\liminf\limits_{t\to+\infty}\arg(1-\bar{\sigma}\gamma(t))=-\frac{\pi}{2}$ (i.e. there exists a strictly increasing sequence $\{t_n\}\subset[0,+\infty)$ with $\lim\limits_{n\to+\infty}t_n=+\infty$ such that $\{\gamma(t_n)\}$ converges to $\sigma$ by angle $\pi$).
	\end{enumerate}
	Combining Theorems 1.1 and 1.2 with Corollary 4.1, we directly get one final corollary.
	\begin{cor}
		Let $(\phi_t)$ be a non-elliptic semigroup in $\mathbb{D}$ with Koenigs function $h$. Suppose that $\gamma:[0,+\infty)\to\mathbb{D}$ is a regular backward orbit for $(\phi_t)$. Then, the following are equivalent:
		\begin{enumerate}
			\item[\textup{(i)}] $\gamma$ converges non-tangentially to a point of the unit circle,
			\item[\textup{(ii)}] $\gamma$ is a quasi-geodesic for the hyperbolic distance $k_\mathbb{D}$ of the unit disk,
			\item[\textup{(iii)}] there exists a constant $c\ge1$ such that
			$$\frac{1}{c}\cdot\delta_\Omega^-(h(\gamma(t)))\le\delta_\Omega^+(h(\gamma(t)))\le c\cdot\delta_\Omega^-(h(\gamma(t))),$$
			for all $t\ge0$.
		\end{enumerate}
	\end{cor}
	Analogous results with the above hold for forward trajectories as well (see \cite[Corollary 17.3.2, Theorem 17.3.3]{book}).
	\begin{rem}
		Observing all the proofs of all the aforementioned theorems and corollaries, it is easily understood that any difficulty is mostly due to super-repelling fixed points. For example, Corollary 4.1 (or Corollary 4.2) might seize to hold if we let $\gamma$ converge tangentially (or non-tangentially respectively) to a super-repelling fixed point. A natural conjecture is that both corollaries can be generalized for such backward orbits as well.
	\end{rem}

\section*{Acknowledgments}
\par I thank Professor D. Betsakos for the helpful conversations during the preparation of this work.
\par I also thank Professor M. D. Contreras for the helpful correspondence concerning quasi-geodesics.

\end{document}